\documentclass[10pt]{article}

\usepackage{enteteanglais}

\title{A note on degenerations of Morse actions\footnote{\textbf{AMS Codes} : 22E40 53C35 20E42 \textbf{Keywords} : Representations of discrete groups in semi-simple Lie groups, degenerations Morse actions on symmetric spaces and buildings.}}

\author{Louis Merlin}

\begin{document}

\maketitle

\begin{abstract}
We study Morse representations of discrete subgroups in higher rank semi-simple Lie groups defined in \cite{klp2}. We show that, if a sequence of Morse representations $\rho_n:\Gamma\rightarrow G$ is (strongly) unbounded in the character variety, the group must have a very particular structure : a free product of a free group and surface groups.
\end{abstract}

\section{Introduction}

In \cite{klp1}, M. Kapovich, B. Leeb and J. Porti introduced a class of discrete subgroups of higher rank Lie groups, called the Morse subgroups, which enjoy remarkable "rank one properties". Those groups are meant to replace the rank one convex cocompact subgroups for which the naive definition (the convex core is compact) fails to produce interesting examples in higher rank (\cite{quintcc} and \cite{kleinerleebcc}). The authors used opportune characterizations of convex cocompactness (using dynamical properties on the boundaries or coarse geometric properties) to generalize it in higher rank. The terminology refers to the fact that those groups satisfy a higher rank version of the Morse lemma, see \cite{klp2}. We postpone in section \ref{prel} a formal definition and a brief discussion on the propreties of Morse subgroups and actions. Those groups have been paid with a lot of attention, in particular in the direction of Anosov representations; see \cite{ggkw}, \cite{anosovhyperbolic} and \cite{klpsurvey}.

As being thought of groups that exhibit rank one behaviour properties, it is natural to ask what can be said about degenerations of Morse representations. Indeed, there is a huge variety of different phenomena in rank one : lattices are rigid in dimensions greater than 2 (Mostow’s rigidity theorem) while flexible in dimension 2 (Teichmuller space). The representations of fundamantal group of a hyperbolic acylindrical 3-manifold inside PSL$_2(\mathbb{C})$ is compact (\cite{cullershalen}), while Shottky groups or quasi-fuchsian representations have a non compact space of representations (actually the statement below will precise this observation).

\paragraph*{Main Result.}

In the sequel, $G$ will be a semi-simple Lie group with no compact factor. We denote by $(X,g)$ the Riemannian symmetric space associated to $G$, that is the space $G/K$ where $K$ is a maximal compact subgroup, endowed with a $G$-left invariant metric. The reader who is not familiar with the subject may keep in mind the case $X=\mathbb{H}^2\times\mathbb{H}^2$, $G=$PSL$_2(\mathbb{R})\times$PSL$_2(\mathbb{R})$. For an element $\gamma\in G$, we denote by $\norm{\gamma}$ its translation length, that is
\[\norm{\gamma}=\inf_{x\in X} d(\gamma x,x),\]
where $d$ is the distance derived from $g$.
We only consider actions of a finitely generated group $\Gamma$ by Riemannian isometries on $X$.

\begin{defi}[tame degeneration]\label{defi:tame}
Let $\rho_n : \Gamma \rightarrow G$ be a sequence of representations of a finitely generated group $\Gamma = <s_1,\cdots,s_k>$ into a Lie group $G$. We set 
\[\lambda_n = \inf_i \{\norm{\rho_n(s_i)}\}\quad\quad \mathrm{and}\quad\quad \Lambda_n = \max_i \{\norm{\rho_n(s_i)}\}\]
We say that $\rho_n$ tamely degenerates if the sequence $(\Lambda_n)$ is unbounded and if there exists a constant $A$ such that
\[A\Lambda_n\leqslant\lambda_n.\]
Let us emphasize the fact that tameness does depend on a generating set. However, to avoid heaviness, we will say that a sequence of representations tamely degenerates if \emph{there exists a generating set} for which the sequence is tame.
\end{defi}

The following main result is an algebraic obstruction for existence of tamely degenerating Morse representations.

\begin{thm}[Main result]
Let $\Gamma$ be a finitely generated subgroup and let $G$ be a semi-simple Lie group with no compact factors. Assume $\rho_n$ is a sequence of uniform Morse representations $\rho_n : \Gamma \rightarrow G$ which tamely degenerates. Then $\Gamma$ splits as a free product of a free group and closed surface groups. 
\end{thm}

The word "uniform" refers to the parameters needed to define a Morse representation (see paragraph \ref{subsection:morse} and remark \ref{metrics} below). Even though we are not able to describe a frank dividing line between uniformly and non uniformly Morse degenerations, we know that examples of uniformly Morse degenerations do exist, for instance for Schottky groups.

The rough conclusion of this statement is that degenerating Morse groups have a structure which looks like Fuchsian groups. This corroborates the general phenomenon of groups enjoying rank one behaviour, in the context of representations.\\

The hypothesis of strong degeneracy (that is $\inf_i \{\norm{\rho_n(s_i)}\}$ must be unbounded) could seem a little strong. However such degenerations are not that exotic. Indeed, consider the situation of a closed surface group $\Gamma$ and a sequence of reprensentations of $\Gamma$ which goes to infinity in the Teichmüller space. It is not hard to show that there exist generators for $\Gamma$ such that all the translation lengths go to infinity. To do that, we need to use Thurston's description of the boundary of Teichmuller space, by geodesic laminations. Given such a boundary lamination, take generators of $\Gamma$ which all intersect the lamination. Their translation lengths must diverge. Building upon those ideas we can also show that we can find generators which diverge at the same rate, hence producing a tame degeneration.

Nevertheless, to the knowledge of the author, we cannot make such a choice for degenerations of Morse subgroups : there is no analogous of Thurston's boundary for higher Teichmüller spaces.

Before going any further, let's make a couple of remarks to illusrate the relevance of this statement.

\begin{rmq}
\emph{
\begin{enumerate}
\item Consider a uniform lattice $\Gamma$ in a higher rank semi-simple Lie group $G$ (the representation $\Gamma\rightarrow G$ is then not at all flexible by Mostow's rigidity theorem). Then $\Gamma$ cannot be a free product of surface groups and a free group. Indeed those groups are Gromov-hyperbolic and they cannot be uniform lattices of higher ranks Lie groups.
\item One of the ingredients of the definition of Morse action is that the orbit map
\[\Gamma\rightarrow \Gamma\cdot x\]
is a quasi-isometry. This is also true for a cocompact lattice by Milnor-Schwarz lemma. This shows that the full hypothesis of Morse action is necessary (see however remark \ref{hyperbolic}).
\end{enumerate}
}
\end{rmq}

\paragraph*{Outline of the proof and plan of the paper.}

Here are two heuristical facts that will help us to build a solid plan. Let us make the same hypothesis than in the main theorem.\\

\textbf{Fact 1:} Going at infinity in the space of representations amounts to consider an action of $\Gamma$ on a Euclidean $\mathbb{R}$-building which is a singular CAT$(0)$-space having some combinatorial similarities with the symmetric space associated to $G$. This space is called the asymptotic cone of the symmetric space. This is a direct reference to the construction of F. Paulin and A. Parreau in \cite{paulindegenerescence} and \cite{parreau} (for a surface group this is also known as the Bestvina-Paulin compactification of the Teichmuller space, which was described in an unpublished part of Paulin's dissertation and in \cite{bestvina}). See paragraph \ref{geometry}.

\textbf{Fact 2:} When one wants to study the structure of some group, it is more eloquent to let it act on such a building than on a Riemannian manifold. An evidence for that is the Rips' theorem (\cite{gaboriaulevittpaulin}) which will be decisively used at the end of the proof. Rips' theorem is stated as theorem \ref{rips}.\\

Keeping those two guiding facts in mind, we proceed as follows:\\

\textbf{Step 1:} Let $(X,g)$ be the Riemannian symmetric space associated to $G$. Recall we denoted by $
\Lambda_n$ the number
\[\Lambda_n = \max_i \{\norm{\rho_n(s_i)}\}.\]
We can assume that $\Lambda_n\to \infty$. We consider the sequence of symmetric spaces $(X,g_n)$ where
\[g_n=\frac{g}{\sqrt{\Lambda_n}}.\]
This sequence converges for the equivariant Gromov-Hausdorff topology (\cite{paulindegenerescence}) to an Euclidean $\mathbb{R}$-building $\chi$ naturally endowed with an action of $\Gamma$. The choices of some basepoints and a principal ultrafilter is harmless.

\textbf{Step 2:} Define a notion of Morse action on a building.

\textbf{Step 3:} Show that the limit action of $\Gamma$ on $\chi$ is Morse.

\textbf{Step 4:} Show that the limit action is free.

\textbf{Step 5:} Generalize Rip's theorem for a free Morse action on a $\mathbb{R}$-building.\\

Steps 1 and 2 come from classical facts about asymptotic cones and Morse actions and will be proved in the course of section \ref{prel}. The 3 last steps are more specific and will be proved in section \ref{proof}.

\paragraph*{Aknowledgments.} I would like to heartily thank Mahan MJ for suggesting me the problem and for explaining me the usefullness of the asymptotic cone. My interest for the subject and the understanding of the paper \cite{klp1} owes a lot to discussions we had with Yash Lodha.

I also thank Mahan MJ for informing me that M. Kapovich has independently worked out a similar proof (but never written up).

\section{Preliminaries}\label{prel}

\subsection{Symmetric spaces of noncompact type, buildings and asymptotic cones of symmetric spaces}\label{geometry}

The literature about symmetric spaces and buildings is quite abundant. Complete and solid references are given in the books \cite{eberlein} for symmetric spaces and \cite{buildingsbook} for buildings. Regarding symmetric spaces, we also refer to paragraph 2.3 of \cite{klp1} for an efficient overview. For buildings, the spirit of this article is closer to the geometric definition given in chapters 3 and 4 of \cite{kleinerleebihes}. Indeed the buildings we will consider are $\mathbb{R}$-buidings which are in general nor discrete, nor Bruhat-Tits buildings associated to some algebraic group.

Let us just recall the construction of the asymptotic cone given in \cite{kleinerleebihes} and replaced in the context of degenerations of discrete groups in \cite{paulindegenerescence} and \cite{parreau}. Indeed, let $(X,d)$ be a metric space. We will apply this construction in only two cases: a symmetric space with Riemannian distance $d$ and a group $\Gamma$ with a word metric. Let also $(\Lambda_n)_{n\in\mathbb{N}}$ be a sequence of positive real numbers such that
\[\lim_{n\to\infty}\Lambda_n=+\infty.\]
According to this real numbers, we consider the scaled metric spaces $(X_n,d_n)$ where, for any $n$,
\[X_n= X \quad \mbox{ and } \quad d_n=\frac{g}{\Lambda_n}\]
We also fix a nonprincipal ultrafilter $\omega$ and a sequence of basepoint $\star_n\in X_n$. Now we construct the space
\[\chi_\infty=\set{x=(x_n)_n\in \prod_{i=1}^\infty X_n \mbox{ such that } \lim_\omega d_n(x_n,\star_n)<\infty}.\]
It is endowed with a pseudo-metric
\[d_\infty\left((x_n),(y_n)\right)=\lim_\omega d_n(x_n,y_n)\]
which is not a metric in general. Hence the space we will merely consider is defined as the quotient of $\chi_\infty$ by identifying two points at $d_\infty$-distance $0$. It is a (pointed) metric space denoted by $(\chi,d_\infty,\crochet{(\star_n)})$ where $\crochet{(x_n)}$ denotes the class in $\chi$ of an element $(x_n)\in \chi_\infty$.

Returning to the context of degenerating representations of some discrete group $\Gamma$ in $G$, with associated symmetric space $X=G/K$, the usefullness of the asymptotic cone is explained by the three statements below.

\begin{lem}[Lemma 2.4.3 in \cite{kleinerleebihes} adapted to the contex of representations in \cite{paulindegenerescence}]\label{lem:convergence}
The pointed metric spaces $(X_n,d_n,\star_n)$ converges to $(\chi,d_\infty,\crochet{(\star_n)})$ for the equivariant Gromov-Hausdorff topology.
\end{lem}

\begin{thm}[Theorem 5.2.1 in \cite{kleinerleebihes}]\label{thm:structurelimitbuilding}
The metric space $(\chi,d_\infty)$ has a structure of a Euclidean $\mathbb{R}$-building of the same type as $X$ (that is the Weyl group of $\chi$ is a group of affine isometries whose vectorial parts lie in the Weyl group of $X$).
\end{thm}

Beware that $\chi$ is a $\mathbb{R}$-buiding in general which means that its Weyl group may not be discrete.

\begin{lem}\label{action}
We have an action of $\Gamma$ on $\chi$ given by
\[\fonction{\rho}{\Gamma\times\chi}{\chi}{(\gamma,\crochet{(x_n)})}{\crochet{(\gamma x_n)}}.\]
\end{lem}

\begin{proof}
The group $\Gamma$ acts on each $(X_n,d_n)$ by isometries, hence acts on $\chi_\infty$ by
\[\gamma\cdot (x_n) = (\gamma x_n)\]
and by isometries. So it goes down to $\chi$.
\end{proof}

\subsection{Morse actions}\label{subsection:morse}

This part is made with selected pieces of \cite{klp1}. Before giving the definition of Morse groups and actions, we need to describe some combinatorial features of symmetric spaces and buildings (still assuming that the reader is familiar with the basics of the theory). For the rest of this section the metric space $(X,d)$ could be either a symmetric space or an Euclidean $\mathbb{R}$-building. We endow the boundary at infinity $\partial_\infty X$ with the Tits metric (section 3.4 in \cite{eberlein}) and denote it by $\partial_\mathrm{Tits}X$. It has a structure of a spherical building (3.6 in \cite{eberlein}).

We denote by $\sigma_{\mathrm{mod}}$ the model Weyl chamber at infinity, that is
\[\sigma_{\mathrm{mod}}=\partial_\infty X /G\]
and the projection $\theta:\partial_\infty X \rightarrow \partial_\infty X /G$ which is called the type map. Tha map $\theta$ restricts to an isometry on every simplex of $\partial_\mathrm{Tits}X$. This allows to also define the type of a simplex. We use the notation $\theta(\tau)=\tau_{\mathrm{mod}}$. The type of a simplex is then a face of $\sigma_{\mathrm{mod}}$.

Throughought this text, we fix an additional data $\Theta$ which is a compact set of $\sigma_{\mathrm{mod}}$ and which encodes the type of regularity of vectors and boundary points. Morse actions will be defined with respect to this choice.

\paragraph*{Stars and regularity.}

Let $\tau\subset \partial_\mathrm{Tits}X$ be a simplex. The star of $\tau$, denotes st$(\tau)$ is the union of all simplices intersecting int$(\tau)$ non trivially. The open star of $\tau$, denoted ost$(\tau)$ consists in all open simplices whose closure intersects int$(\tau)=\overset{\circ}{\tau}$ non trivially.

Those notions of stars allow to generalize the definition of regularity for a point $\xi\in\partial_\infty X$ (hence also for a vector $v\in TX$ or a geodesic inside $X$) in \cite{eberlein} 2.11. Precisely, a point $\xi\in\partial_\infty X$ such that $\theta(\xi)\in$ ost$(\tau_{\mathrm{mod}})$ is called $\tau_{\mathrm{mod}}$-regular. If $\Theta$ is a compact subset of $\sigma_{\mathrm{mod}}$ and $\xi\in\Theta$, we say that $\xi$ is $\Theta$-regular.

Note that, if $\tau$ is of type $\sigma_{\mathrm{mod}}$, then st$(\tau)=\tau$ and ost$(\tau)=\overset{\circ}{\tau}$. But the star is not finite as soon as $\overset{\circ}{\tau}$ contains a singular point (for instance, the endpoint of a geodesic $\gamma(t)=(o,c(t))$ in $\mathbb{H}^2\times \mathbb{H}^2$ where $c$ is a non constant geodesic of the hyperbolic plane is contained in many chambers).

A segment in $X$ is called $\tau_{\mathrm{mod}}$ or $\Theta$-regular if it is contained in a geodesic $x\xi$ such that $\theta(\xi)\in$ ost$(\tau_{\mathrm{mod}})$ or $\theta(\xi)\in\Theta$ respectively.

\paragraph*{Cones.}

Let $x$ be a point in $X$ and $\tau$ be a simplex of $\partial_\mathrm{Tits}X$. The Weyl sector, denoted V$(x,\tau)$ is the union of all geodesic rays $x\xi$ where $\xi\in\tau$. It is contained in a flat (or an appartment). More generally, if $A$ is closed subset of $\partial_\infty X$, the Weyl cone, denoted V$(x,A)$ is the union of all geodesic rays $x\xi$ such that $\xi\in A$.

\paragraph*{Parallel sets.}

Let $s\subset \partial_\mathrm{Tits}X$ be an isometrically embedded unit sphere. The parallel set $P(s)$ is the union of all maximal flats or appartments asymptotic to $s$. For a pair of antipodal simplices $\tau_-$ and $\tau_+$., there is a unique minimal sphere $s(\tau_-,\tau_+)$ containing both of them and we denote by $P(\tau_-,\tau_+)$ the parallel set $P(s(\tau_-,\tau_+))$.

\paragraph*{Diamonds.}

We recall that we have fixed a compact subset $\Theta$ of $\sigma_{\mathrm{mod}}$. We first define the $\Theta$-star of a simplex $\tau \subset \partial_\mathrm{Tits}X$ as
\[\mathrm{st}_{\Theta}(\tau)= \mathrm{st}(\tau)\cap \theta^{-1}(\Theta)\]
(we construct the star only with simplices of type in $\Theta$). Let $xy$ be a geodesic segment. We define the simplex $\tau(xy)$ as follows: forward extend the geodesic segment to a whole ray $x\xi$, say of type $\tau_{\mathrm{mod}}$-regular. Then $\tau$ is the unique face of type $\tau_{\mathrm{mod}}$ such that $\xi\in$ st$(\tau)$. Finally if $x_-x_+$ is a geodesic which is $\Theta$-regular, we define the $\Theta$-diamond of $x_-x_+$ as
\[
\Diamond_\Theta(x_-x_+) = V(x_-,\mathrm{st}_\Theta(\tau_+))\cap V(x_+,\mathrm{st}_\Theta(\tau_-))
\]
where $\tau_+=\tau(x_-x_+)$ and $\tau_-=\tau(x_+x_-)$. The diamond is a subset of $P(\tau_-,\tau_+)$ (it is indeed shaped like a diamond, being the intersection of two cones going in opposite directions).

\paragraph*{Morse quasi-geodesics.}

We now fix three positive real constants $L$, $A$, $D$. We define a $(L,A,\Theta,D)$-Morse quasi-geodesic as a map $p:I\rightarrow X$ where $I$ is an interval of $\mathbb{R}$ such that $p$ is a $(L,A)$ quasi-geodesic and for all $t_1$, $t_2$ in $I$, the subpath is uniformly $D$-close to a $\Theta$-diamond $\Diamond_\Theta(x_1x_2)$ with $d(x_i,p(t_i))\leq D$ for $i=1,2$.

Such Morse quasi-geodesics appear as quasi-geodesics that satisfy a higher rank version of the Morse lemma (theorem 7.2 in \cite{klp1}).

\paragraph*{Morse actions.}

A $\Theta$-Morse embedding from a quasi-geodesic $Z$ to $X$ is a map $f:Z\rightarrow X$ which sends uniform quasi-geodesics to uniform Morse quasi-geodesics with respect to a diamond of type $\Theta$ (uniform refers to the constant $D$). Let now be $\Gamma$ a discrete group acting by isometries on $X$. We say that the action is Morse if the orbit maps $\Gamma\rightarrow \Gamma\cdot x$ is a Morse quasi-geodesic embedding for a (hence any) word distance on $\Gamma$.

\begin{rmq}\label{hyperbolic}
\emph{
\begin{enumerate}
\item On of the great achievment in \cite{klp2} is to show that a subgroup of $G$ is a Morse subgroup if and only if it is asymptotically embedded (see definition 1.5). In particular a Morse subgroup is a  word-hyperbolic group (corollary 6.17).
\item Even though we stated the main theorem with a Morse condition, the conclusion actually still holds if we only assume that the group is hyperbolic and sends quasi-geodesics to quasi-geodesics, which is a slightly weaker hypothesis.
\end{enumerate}}
\end{rmq}

\section{Proof of the main theorem}\label{proof}

Notations are the same as in section \ref{geometry}. We now come to the proof of the main result

\begin{theo} Let $\rho_n :\Gamma\rightarrow G$ be a tamely unbounded sequence of $(L,A,\Theta,D)$-Morse representations of a finitely generated torsion free subgroup. Then $\Gamma$ splits as a free product of closed surface groups and a free group.
\end{theo}

Fom now on, we assume that the scaling parameters used to construct the asymptotic cone are the numbers $\Lambda_n=\max_i\set{\norm{\rho_n(s_i)}}$ where $(s_i)$ is a generating set for $\Gamma$.

\begin{rmq}[About metrics]\label{metrics}
\emph{
The quasi-isometry (uniform) constants $(L,A)$ are computed with a word metric in $\Gamma$ on one side and for rescaled metrics $\frac{d}{\Lambda_n}$ on the other side as in paragraph \ref{geometry}.}
\end{rmq}

\begin{prop}\label{prop:morselimit}
Let $\rho_n :\Gamma\rightarrow X$ be a degenerating sequence of $(L,A,\Theta,D)$-Morse representations. Then the limit action (see lemma \ref{action}) on the asymptotic cone is $(L,A,\Theta,D)$-Morse.
\end{prop}

\begin{rmq}
\emph{
Only a weaker version of this lemma will be used in the rest of the proof. Precisely, it is enough for our purpose to show that the orbit maps send quasi-geodesics to quasi-geodesics. But the limiting Morse property may be of independent interest and will be proved with the same argument.}
\end{rmq}

\begin{proof}
From lemma \ref{lem:convergence} and from the very definition of Gromov-Hausdorff, we know that $(L,A)$ quasi-geodesics converge to $(L,A)$ quasi-geodesics. We only need to show that the Morse condition is preserved at the limit. By lemma \ref{lem:convergence} again we know that sequences of diamonds in $X_n$ converge to diamonds in $\chi$ of the same type $\Theta$. Indeed, geodesics converge to geodesics of the same type (endpoints of the geodesics must converge and we conclude using that $\Theta$ is compact, so the $\Theta$-star is closed). Then Weyl cones converge to Weyl cones (see also \cite{kleinerleebihes} proposition 2.4.6 for the regular case). So let $\varepsilon$ be a positive real number and $s_{i_1},\cdots,s_{i_k}$ be a quasi-geodesic in $\Gamma$. For any sequence of basepoints $(x_n)$, we get a Morse quasi-geodesic $\gamma_n=s_{i_1}(x_n)\cdots s_{i_k}(x_n)$ in $X_n$. Hence there exists a $\Theta$-diamond $\Diamond_{\Theta,X_n}(a_n,b_n)$ for which
\[d_n(s_{i_1}(x_n),a_n)\leq D,\quad d_n(s_{i_k}(x_n),b_n)\leq D \]
and $\gamma_n$ is uniformly $D$-close (for $d_n$) to $\Diamond_{\Theta,X_n}(a_n,b_n)$.
Now, we can choose $n$ big enough so that $\Diamond_{\Theta,X_n}(a_n,b_n)$ is $\varepsilon$-close to some diamond $\Diamond_{\Theta,\chi}(a,b)$ in $\chi$ and we note that this gives a $(L,A,\Theta,D+\varepsilon)$-Morse action of $\Gamma$ on $\chi$ and $\varepsilon$ is arbitrary small.
\end{proof}

\begin{lem}\label{lem:freeaction}
Let $\Gamma = <s_1,\cdots,s_k>$ be a finitely generated group acting Morsely (via $\rho$) on a $\mathbb{R}$-building $\chi$. Suppose none of the generators have a fixed point. Then the action is free.
\end{lem}

\begin{proof}
Fronm the hypothesis on the genertors, we only keep in mind that none of the $\rho(s_i)$ is trivial.
Let $x\in \chi$ and let $g\in \Gamma$. It is enough to show that $g^k$ doesn't fix $x$ for some $k$ sufficiently big. Since the group $\Gamma$ is torsion free, the $g^k$ are all different. Hence, for some $k$ big enough, $g^k$ is not contained in the ball of radius $AL$ around the identity element for the word distance $\delta$ in $\Gamma$. We get, for such a choice of $k$,
\[d(g^kx,x)\geqslant L^{-1}\delta(g^k,id)-A>0.\]
\end{proof}

The condition of tame degeneration precisely ensures that none of the generators have a fixed point when acting on $\chi$. Indeed
\begin{eqnarray*}
d_\infty(\crochet{(x_n)},s_i\crochet{(x_n)}) & = & \lim_\omega d_n(x_n,\rho_n(s_i)x_n)\\
 & = & \lim_\omega \frac{d_n(x_n,\rho_n(s_i)x_n)}{\Lambda_n}\\
 & \geqslant & \lim_\omega \frac{\lambda_n}{\Lambda_n}\\
 & \geqslant & A
\end{eqnarray*}
(recall definition \ref{defi:tame})

\begin{lem}
Let $\Gamma = <s_1,\cdots,s_k>$ be a finitely generated group acting freely and Morsely on a $\mathbb{R}$-building $\chi$. Then $\Gamma$ splits as a free product of a free group and closed surface groups. 
\end{lem}

The goal is to apply Rips' theorem below.

\begin{thm}[Rips' theorem, see \cite{gaboriaulevittpaulin}]\label{rips}
Let $\Gamma$ be a finitely generated group acting freely on a (possibly non discrete) tree. Then $\Gamma$ is a free product of free Abelian groups with closed surface groups.
\end{thm}

\begin{proof}
We need to find a tree on which $\Gamma$ acts freely. By a result of M. Gromov in \cite{gromovasymptotictrees}, part 2.6 (see also \cite{drutuasymptotictrees}) and because $\Gamma$ is hyperbolic (remark \ref{hyperbolic}), the asymptotic cone of $(\Gamma,\delta)$ where $\delta$ is any word metric (and for any choice of basepoints and a principal ultrafilter $\omega$) is a $\mathbb{R}$-tree. We denote this metric space by $(T,\delta)$. Basepoints in $X_n$ are still denoted $\star_n$ amd the ultrafilter is the same.

Moreover, since the embedding $f_n : (\Gamma,\delta) \rightarrow (\Gamma\cdot \star_n,d_n)$ is a $(L,A)$ quasi-isometry, the map
\[\fonction{f_\omega}{T}{\chi}{\crochet{(\gamma_n)_n}}{\crochet{(f_n(\gamma_n))_n}=\crochet{(\gamma_n(x_n))_n}}\]
is a bi-lipschitz embedding (lemma 2.24 in \cite{klp1}).
We get that $\Gamma$ acts on the bilipschitz $\mathbb{R}$-tree $f_\omega(T)$ by
\[\gamma\cdot \crochet{(\gamma_n(x_n))_n} = \crochet{(\gamma\gamma_n(x_n))_n}\]
This action is free since it is the restriction of the action described in lemma \ref{action} and proved to be free in lemma \ref{lem:freeaction}.

Finally, taking into account the hyperbolicity of the group, each free Abelian factor must be either trivial or $\mathbb{Z}$.
\end{proof}

\bibliographystyle{alpha}
\bibliography{biblio}

\end{document}